\documentclass{article}

\usepackage{graphicx}
\usepackage{amssymb}
\usepackage{amsmath}
\usepackage{wrapfig}
\usepackage{multicol,caption}
\usepackage{lipsum}

\newtheorem{theorem}{Theorem}
\newtheorem{lemma}{Lemma}
\newtheorem{proposition}{Proposition}
\newtheorem{definition}{Definition}

\newenvironment{proof}{\noindent{\sf Proof.}}{\hfill $\boxtimes\hspace{2mm}$\linebreak}
\newcommand{\qed}{\hfill $\boxtimes\hspace{2mm}$ \linebreak}
\begin{document}

\title{Information Flow in Strategic Games}
\title{Functional Dependence in Strategic Games}
\author{ Kristine Harjes and Pavel Naumov\\ \\ \small Department of Mathematics and Computer Science\\ \small McDaniel College, 
Westminster, Maryland, USA \\ \small
              {\sf \{keh013,pnaumov\}@mcdaniel.edu}     }

\def\titlerunning{Functional Dependence in Strategic Games}
\def\authorrunning{Kristine Harjes and Pavel Naumov}

\maketitle

\begin{abstract}
The paper studies properties of functional dependencies between strategies of players in Nash equilibria of multi-player strategic games. The main focus is on the properties of functional dependencies in the context of a fixed dependency graph for pay-off functions. A logical system describing properties of functional dependence for any given graph is proposed and is  proven to be complete.
\end{abstract}




\section{Introduction}

\vspace{2mm}
\noindent {\bf Functional Dependence.}
In this paper we study dependency between players' strategies in Nash equilibria. For example, the coordination game described by Table~\ref{coordination game} has two Nash equilibria: $(a_1,b_1)$ and $(a_2,b_2)$. Knowing the strategy of player $a$ in  a Nash equilibrium of this game, one can predict the strategy of player $b$. We say that player $a$ functionally determines player $b$ and denote this by $a\rhd b$. 

\begin{wraptable}{l}{0.40\textwidth}
\begin{center}
\vspace{-5mm}
\begin{tabular}{l|c|c|c}
		& $b_1$ 	& $b_2$\\ \hline
$a_1$	& 1,1	& 0,0\\ 
$a_2$	& 0,0	& 1,1\\ 
\end{tabular}
\caption{Coordination Game}
\label{coordination game}
\end{center}
\vspace{5mm}
\end{wraptable}

Note that in the case of the coordination game, we also have $b\rhd a$. However, for the game described by Table~\ref{asymmetric game} statement $a\rhd b$ is true, but $b\rhd a$ is false.

\begin{wraptable}{r}{0.4\textwidth}
\vspace{0mm}
\begin{center}
\begin{tabular}{l|c|c|c}
		& $b_1$ 	& $b_2$\\ \hline
$a_1$	& 1,1	& 0,0\\ 
$a_2$	& 0,0	& 1,1\\ 
$a_3$	& 1,1	& 0,0\\ 
\end{tabular}
\end{center}
\caption{Strategic Game}
\label{asymmetric game}
\vspace{-2mm}
\end{wraptable}

The main focus of this paper is functional dependence in multiplayer games. For example, consider a ``parity" game with three players $a$, $b$, $c$. Each of the players picks 0 or 1, and all players are rewarded if the sum of all three numbers is even. This game has four different Nash equilibria: $(0,0,0)$, $(0,1,1)$, $(1,0,1)$, and $(1,1,0)$. It is easy to see that knowledge of any two players' strategies in a Nash equilibrium reveals the third. Thus, using our notation, for example $a,b\rhd c$. At the same time, $\neg(a\rhd c)$. 

As another example, consider a game between three players in which each player  picks 0 or 1 and all players are rewarded if they have chosen the same strategy. This game has only two Nash equilibria: $(0,0,0)$ and $(1,1,1)$. Thus, knowledge of the strategy of player $a$ in a Nash equilibrium reveals the strategies of the two other players. We write this as $a\rhd b,c$.

Functional dependence as a relation has been studied previously, especially in the context of database theory. Armstrong~\cite{a74} presented the following sound and complete axiomatization of this relation:
\begin{enumerate}
\item 
{\em Reflexivity}: $A\rhd B$, if $B\subseteq A$,
\item 
{\em Augmentation}: $A\rhd B \rightarrow A,C\rhd B,C$,
\item 
{\em Transitivity}: $A\rhd B \rightarrow (B\rhd C \rightarrow A\rhd C)$,
\end{enumerate}
where here and everywhere below $A,B$ denotes the union of sets $A$ and $B$. The above axioms are known in database literature as Armstrong's axioms~\cite{guw09}. 
Beeri, Fagin, and Howard~\cite{bfh77} suggested a variation of Armstrong's axioms that describe properties of multi-valued dependence. 

\vspace{2mm}
\noindent {\bf Dependency Graphs.}
As a side result, we will show that the logical system formed by the Armstrong axioms is sound and complete with respect to the strategic game semantics. Our main result, however, is a sound and complete axiomatic system for the relation $\rhd$ in games with a given dependency graph.

Dependency graphs~\cite{kls01uai, lks01nips, egg07ec,egg06eccc} put restrictions on the pay-off functions that can be used in the game. For example, dependency graph $\Gamma_1$ depicted in Figure~\ref{intro_alpha}, specifies that the pay-off function of player $a$ only can depend on the strategy of player $b$ in addition to the strategy of player $a$ himself.  The pay-off function for player $b$ can only depend on the strategies of players $a$ and $c$  in addition to the strategy of player $b$ himself, etc. 

\begin{wrapfigure}{l}{0.45\textwidth}
\begin{center}
\vspace{-2mm}
\scalebox{.5}{\includegraphics{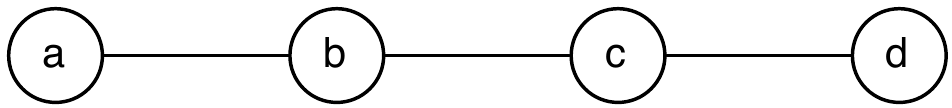}}
\vspace{0mm}
\footnotesize\caption{Dependency Graph $\Gamma_1$}\label{intro_alpha}
\vspace{-5mm}
\end{center}
\vspace{0cm}
\end{wrapfigure}

An example of a game over graph $\Gamma_1$ is a game between players $a$, $b$, $c$, and $d$ in which these players choose real numbers as their strategies. The pay-off function of players $a$ and $d$ is the constant 0. Player $b$ is rewarded if his value is equal to the mean of the values of players $a$ and $c$. Player $c$ is rewarded if his value is equal to the mean of the values of players $b$ and $d$. Thus, Nash equilibria of this game are all quadruples $(a,b,c,d)$ such that $2b=a+c$ and $2c=b+d$. Hence, in this game $a,b\rhd c,d$ and $a,c\rhd b,d$, but $\neg(a\rhd b)$.

Note that although the statement $a,b\rhd c,d$ is true for the game described above, it is not true for many other games with the same dependency graph $\Gamma_1$. In this paper we study properties of functional dependence that are common to all games with the same dependency graph. An example of such statement for the graph $\Gamma_1$, as we will show in Proposition~\ref{XYZ}, is $a\rhd d \rightarrow b,c\rhd d$. 

Informally, this property is true for any game over graph $\Gamma_1$ because any dependencies between players $a$ and $d$ must be established through players $b$ and $c$. This intuitive approach, however, does not always lead to the right conclusion. For example, in graph $\Gamma_2$ depicted in Figure~\ref{intro_beta}, players $b$ and $c$ also separate players $a$ and $d$. Thus, according to the same intuition, the statement $a\rhd d \rightarrow b,c\rhd d$ must also be true for any game over graph $\Gamma_2$. This, however, is not true. Consider, for example, a game in which all four players have three strategies: {\em rock}, {\em paper}, and {\em scissors}. The pay-off function of players $a$ and $d$ is the constant 0. If $a$ and $d$ pick the same strategy, then neither $b$ nor $c$ is paid. If players $a$ and $d$ pick different strategies, then players $b$ and $c$ are paid according to the rules of the standard rock-paper-scissors game. In this game Nash equilibrium is only possible if $a$ and $d$ pick the same strategy. Hence, $a\rhd d$. At the same time, in any such equilibria $b$ and $c$ can have any possible combination of values. Thus, $\neg (b,c\rhd d)$. Therefore, the statement $a\rhd d \rightarrow b,c\rhd d$ is not true for this game.

\begin{wrapfigure}{r}{0.45\textwidth}
\begin{center}
\vspace{-5mm}
\scalebox{.5}{\includegraphics{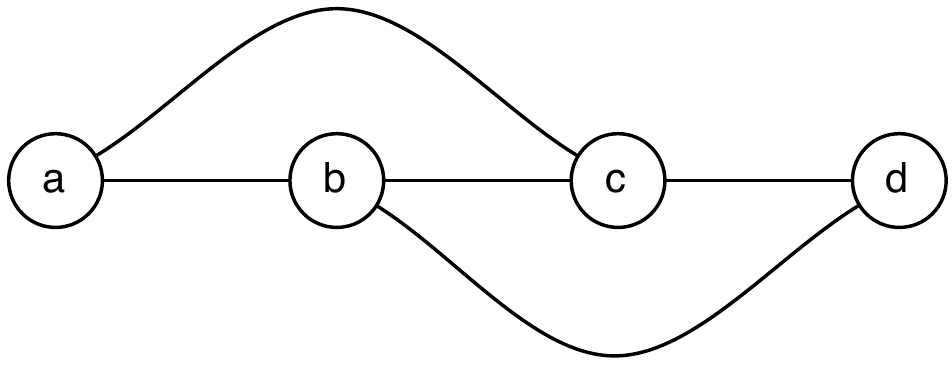}}
\vspace{0mm}
\footnotesize\caption{Dependency Graph $\Gamma_2$}\label{intro_beta}
\vspace{-5mm}
\end{center}
\vspace{0cm}
\end{wrapfigure}

As our final example, consider the graph $\Gamma_3$ depicted in Figure~\ref{intro_gamma}. We will show that $a\rhd c\rightarrow b\rhd c$ is not true for at least one game over graph $\Gamma_3$. Indeed, consider the game in which players $a,b$, and  $c$ use real numbers as possible strategies. Players $a$ and $c$ have a constant pay-off of 0. The pay-off of the player $b$ is equal to $0$ if players $a$ and $c$ choose the same real number. Otherwise, it is equal to the number chosen by the player $b$ himself. Note that in any Nash equilibrium of this game, the strategies of players $a$ and $c$ are equal. Therefore, $a\rhd c$, but $\neg(b\rhd c)$.

\begin{wrapfigure}{l}{0.45\textwidth}
\begin{center}
\vspace{-2mm}
\scalebox{.5}{\includegraphics{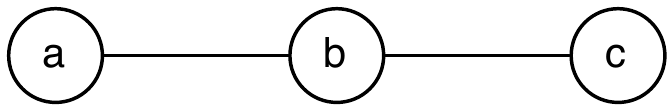}}
\vspace{0mm}
\footnotesize\caption{Dependency Graph $\Gamma_3$}\label{intro_gamma}
\vspace{-5mm}
\end{center}
\vspace{0cm}
\end{wrapfigure}

The main result of this paper is a sound and complete axiomatization of all properties of functional dependence for any given dependency graph. This result is closely related to work by More and Naumov on functional dependence of secrets over hypergraphs~\cite{mn11clima}. However, the logical system presented in this paper is significantly different from theirs. A similar relation of ``rational" functional dependence without any connection to dependency graphs has been axiomatized by Naumov and Nicholls~\cite{nn12loft}. 

The counterexample that we have constructed for the game in Figure~\ref{intro_gamma} significantly relies  on the fact that  player $b$ has infinitely many strategies. However, in this paper we
 show completeness with respect to the semantics of finite games, making the result stronger.

\section{Syntax and Semantics}

The graphs that we consider in this paper contain no loops, multiple edges, or directed edges.
\begin{definition}\label{border}
For any set of vertices $U$ of a graph $(V,E)$, border ${\cal B}(U)$ is the set 
$$\{v\in U\;|\; \mbox{$(v,w)\in E$ for some $w\in V\setminus U$}\}.$$
\end{definition}
A cut $(U,W)$ of a graph $(V,E)$ is a partition $U\sqcup W$ of the set $V$. For any  vertex $v$ in a graph, by $Adj(v)$ we mean the set of all vertices adjacent to $v$. By $Adj^+(v)$ we mean the set $Adj(v)\cup\{v\}$.

\begin{definition}\label{formula}
For any graph $\Gamma=(V,E)$, by $\Phi(\Gamma)$ we mean the minimal set of formulas such that
(i) $\bot\in \Phi(\Gamma)$,
(ii) $A\rhd B\in \Phi(\Gamma)$ for each $A\subseteq V$ and $B\subseteq V$,
(iii) $\phi\rightarrow\psi\in\Phi(\Gamma)$ for each $\phi,\psi\in\Phi(\Gamma)$.
\end{definition}

\begin{definition}\label{}
By game over graph $\Gamma=(V,E)$ we mean any strategic game $G=(V,\{S_v\}_{v\in V},\{u_v\}_{v\in V})$ such that
(i) The finite set of players in the game is the set of vertices $V$,
(ii) The finite set of strategies $S_v$ of any player $v$ is an arbitrary set,
(iii) The pay-off function $u_v$ of any player $v$ only depends on the strategies of the players in $Adj^+(v)$.
\end{definition}
\noindent By $NE(G)$ we denote the set of all Nash equilibria in the game $G$. 
The next definition is the core definition of this paper. The second item in the list below gives a precise meaning of the functional dependence predicate $A\rhd B$.

\begin{definition}\label{true}
For any game $G$ over graph $\Gamma$ and any $\phi\in\Phi(\Gamma)$, we define binary relation $G\vDash \phi$ as follows
(i) $G\nvDash\bot$,
(ii) $G\vDash A\rhd B$ if ${\mathbf s}=_A{\mathbf t}$ implies ${\mathbf s}=_B {\mathbf t}$ for each ${\mathbf s},{\mathbf t}\in NE(G)$,
(iii) $G\vDash\psi_1\rightarrow\psi_2$ if $G\nvDash\psi_1$ or $G\vDash\psi_2$,
where here and everywhere below $\langle s_v\rangle_{v\in V}=_X \langle t_v\rangle_{v\in V}$ means that $s_x=t_x$ for each $x\in X$.
\end{definition}

\section{Axioms}
The following is the set of axioms of our logical system. It consists of the original Armstrong axioms and an additional Contiguity axiom that captures properties of functional dependence specific to a given graph $\Gamma$.

\begin{enumerate}
\item Reflexivity: $A\rhd B$, where $B\subseteq A$
\item Augmentation: $A\rhd B\rightarrow A,C\rhd B,C$
\item Transitivity: $A\rhd B \rightarrow (B\rhd C \rightarrow A\rhd C)$
\item Contiguity: $A,B\rhd C\rightarrow {\cal B}(U),{\cal B}(W),B\rhd C$, where $(U,W)$ is a cut of the graph such that $A\subseteq U$ and $C\subseteq W$.
\end{enumerate}
Note that the Contiguity axiom, unlike the Gateway axiom~\cite{mn11clima}, effectively requires ``double layer" divider ${\cal B}(U),{\cal B}(W)$ between sets $A$ and $C$. This is because in our setting values are assigned to the vertices and  not to the edges of the graph.

We write $\vdash_\Gamma\phi$ if $\phi\in \Phi(\Gamma)$ is provable from the combination of the axioms above and propositional tautologies in the language $\Phi(\Gamma)$ using the Modus Ponens inference rule. We write $X \vdash_\Gamma\phi$ if $\phi$ is provable using the additional set of axioms $X$. We often omit the parameter $\Gamma$ when its value is clear from the context.

\section{Examples}

In this section we give examples of proofs in our formal system. The soundness and the completeness of this system will be shown in the next two sections.

\begin{proposition}\label{XYZ}
$\vdash_{\Gamma_1} a\rhd d\rightarrow b,c\rhd d$, where $\Gamma_1$ is the graph depicted in Figure~\ref{intro_alpha}.
\end{proposition}
\begin{proof}
Consider cut $(U,W)$ of the graph $\Gamma_1$ such that $U=\{a,b\}$ and $W=\{c,d\}$. Thus, ${\cal B}(U)=\{b\}$ and ${\cal B}(W)=\{c\}$. Therefore, by the Contiguity axiom, $a\rhd d\rightarrow b,c\rhd d$.
\end{proof}

\begin{proposition}\label{prop2}
$\vdash_{\Gamma_1} a,c\rhd d\rightarrow (d,b\rhd a\rightarrow b,c\rhd a,d)$, where $\Gamma_1$ is the graph depicted in Figure~\ref{intro_alpha}.
\end{proposition}
\begin{proof}
Assume that $a,c\rhd d$ and $d,b\rhd a$.
Consider cut $(U,W)$ of the graph $\Gamma_1$ such that $U=\{a,b\}$ and $W=\{c,d\}$. Thus, ${\cal B}(U)=\{b\}$ and ${\cal B}(W)=\{c\}$. Therefore,  by the Contiguity axiom with $A=\{a\}$, $B=\{c\}$, and $C=\{d\}$, $a,c\rhd d\rightarrow b,c\rhd d$. Thus, 
\begin{equation}\label{eq1}
 b,c\rhd d.
\end{equation}
 by the first assumption. Similarly, using the second assumption, $b,c\rhd a$. Hence, by the Augmentation axiom, 
\begin{equation}\label{eq2}
b,c\rhd a,b,c.
\end{equation}
Thus, from statement (\ref{eq1}) by the Augmentation axiom,
$a,b,c\rhd a,d$. Finally, using statement~(\ref{eq2}) and the Transitivity axiom,
$b,c\rhd a,d$.
\end{proof}


\begin{proposition}\label{prop3}
$\vdash_{\Gamma_4} a,c\rhd e\rightarrow b,c,d\rhd e$, where $\Gamma_4$ is the graph depicted in Figure~\ref{example_delta}.
\end{proposition}

\begin{wrapfigure}{l}{0.45\textwidth}
\begin{center}
\vspace{-5mm}
\scalebox{.4}{\includegraphics{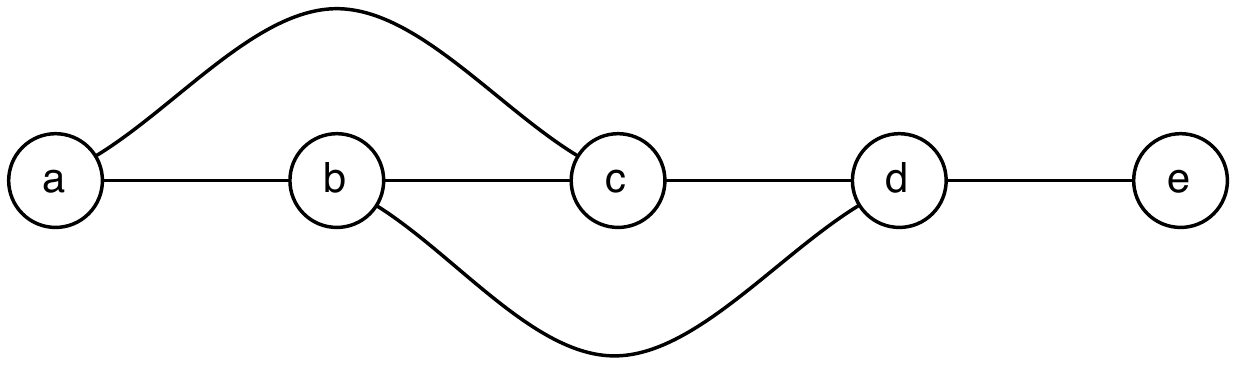}}
\vspace{0mm}
\footnotesize\captionof{figure}{Dependency Graph $\Gamma_4$}\label{example_delta}
\vspace{-6mm}
\end{center}
\vspace{0cm}
\end{wrapfigure}

\noindent{\sf Proof.} Consider cut $(U,W)$ of the graph $\Gamma_4$ such that $U=\{a,b,c\}$ and $W=\{d,e\}$. Thus, ${\cal B}(U)=\{b,c\}$ and ${\cal B}(W)=\{d\}$. Therefore, $a,c\rhd e\rightarrow b,c,d\rhd e$ by the Contiguity axiom with $A=\{a\}$, $B=\{c\}$, and $C=\{e\}$.
\qed
\vspace{-5mm}

\begin{proposition}\label{prop4}
$\vdash_{\Gamma_5} a\rhd b\rightarrow (b\rhd c\rightarrow (c\rhd a \rightarrow d,e,f\rhd a,b,c))$, where
$\Gamma_5$ is depicted in Figure~\ref{example_epsilon}.
\end{proposition}

\noindent{\sf Proof.}
Assume $a\rhd b$, $b\rhd c$, and $c\rhd a$.
Consider cut $(U,W)$ of the graph $\Gamma_5$ such that $U=\{c,f\}$ and $W=\{a,b,d,e\}$. Thus, ${\cal B}(U)=\{f\}$ and ${\cal B}(W)=\{d,e\}$. Therefore, by the Contiguity axiom with $A=\{c\}$, $B=\varnothing$, and $C=\{a\}$, $c\rhd a\rightarrow d,e,f\rhd a$. Hence, $d,e,f\rhd a$ by the third assumption. Similarly, one can show $d,e,f\rhd b$, and $d,e,f\rhd c$. By applying the Augmentation axiom to the last three statements,
$
d,e,f\rhd a,d,e,f,
$
and
$
a,d,e,f\rhd a,b,d,e,f,
$
and
$
a,b,d,e,f\rhd a,b,c.
$
Therefore, $d,e,f\rhd a,b,c$ by the Transitivity axiom applied twice.
\qed

\begin{wrapfigure}{r}{0.45\textwidth}
\begin{center}
\vspace{-2mm}
\scalebox{.5}{\includegraphics{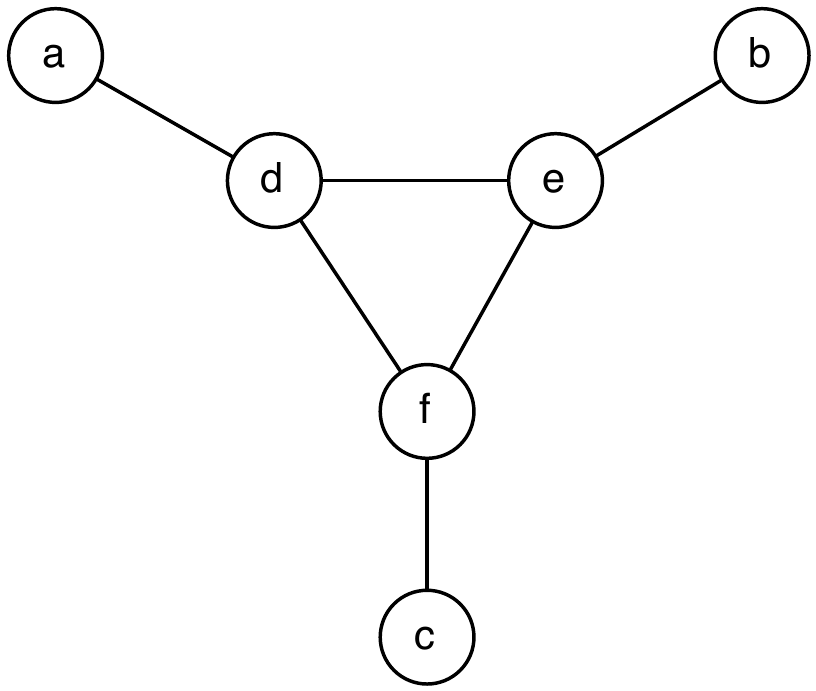}}
\vspace{0mm}
\footnotesize\caption{Dependency Graph $\Gamma_5$}\label{example_epsilon}
\vspace{0mm}
\end{center}
\vspace{0cm}
\end{wrapfigure}

Proposition~\ref{prop2} and Proposition~\ref{prop4} are special cases of a more general principle. We will say that a subset of vertices is {\em sparse} if the shortest path between any two vertices in this subset contains at least three edges. The general principle states that if $W$ is a sparse subset of vertices in the graph $(V,E)$ and each vertex $w\in W$ is functionally determined by the set $V\setminus\{w\}$, then the subset $V\setminus W$ functionally determines the subset $W$:    
$$
\bigwedge_{w\in W}\left((V\setminus \{w\})\right)\rhd w \rightarrow (V\setminus W)\rhd W.
$$
For example, the set $\{a,d\}$ in the graph $\Gamma_1$ depicted in Figure~\ref{intro_alpha} is sparse. Due to the general principle, 
$
a,b,c\rhd d \rightarrow(d,c,b\rhd a \rightarrow b,c\rhd a,d).
$
Thus, by Lemma~\ref{left mono},
$
a,c\rhd d \rightarrow(d,b\rhd a \rightarrow b,c\rhd a,d),
$
which is the statement of Proposition~\ref{prop2}. In the case of Proposition~\ref{prop4}, the sparse set is $\{a,b,c\}$. The proof of the general principle is similar to the proof of Proposition~\ref{prop4}.

\section{Soundness}

We prove soundness of our logical system by proving soundness of each of our four axioms separately.

\begin{lemma}[reflexivity]\label{} 
$G\vDash A\rhd B$ for each game $G$ over a graph $\Gamma=(V,E)$
and each $B\subseteq A\subseteq V$.
\end{lemma}
\begin{proof}
For any ${\mathbf s},{\mathbf t}\in NE(G)$, if ${\mathbf s}=_A {\mathbf t}$, then ${\mathbf s}=_B {\mathbf t}$ because $A\subseteq B$.
\end{proof}

\begin{lemma}[augmentation]\label{} 
If $G\vDash A\rhd B$, then $G\vDash A,C\rhd B,C$ for each game $G$ over a graph $\Gamma=(V,E)$
and each $A,B,C\subseteq V$.
\end{lemma}
\begin{proof}
Suppose that $G\vDash A\rhd B$ and consider any ${\mathbf s},{\mathbf t}\in NE(G)$ such that ${\mathbf s}=_{A,C}{\mathbf t}$. We will show that ${\mathbf s}=_{B,C}{\mathbf t}$. Indeed, ${\mathbf s}=_{A,C}{\mathbf t}$ implies that ${\mathbf s}=_{A}{\mathbf t}$ and ${\mathbf s}=_{C}{\mathbf t}$. Thus, ${\mathbf s}=_{B}{\mathbf t}$ by the assumption $G\vDash A\rhd B$. Therefore, ${\mathbf s}=_{B,C}{\mathbf t}$.
\end{proof}

\begin{lemma}[transitivity]\label{}
If $G\vDash A\rhd B$ and $G\vDash B\rhd C$, then $G\vDash A\rhd C$ for each game $G$ over a graph $\Gamma=(V,E)$
and each $A,B,C\subseteq V$.
\end{lemma}
\begin{proof}
Suppose that $G\vDash A\rhd B$ and $G\vDash B\rhd C$. Consider any ${\mathbf s},{\mathbf t}\in NE(G)$ such that ${\mathbf s}=_{A}{\mathbf t}$. We will show that ${\mathbf s}=_{C}{\mathbf t}$. Indeed, ${\mathbf s}=_{B}{\mathbf t}$ due to the first assumption. Hence, by the second assumption, ${\mathbf s}=_{C}{\mathbf t}$.
\end{proof}

\begin{lemma}[contiguity]\label{}
If $G\vDash A,B\rhd C$, then $G\vDash {\cal B}(S),{\cal B}(T),B\rhd C$,
for each game $G=(V,E)$ over a graph $\Gamma$, each cut $(U,W)$ of $\Gamma$, and each $A\subseteq U$, $B\subseteq V$, and $C\subseteq W$.
\end{lemma}
\begin{proof}
Suppose that $G\vDash A,B\rhd C$. Consider any ${\mathbf s}=\langle s_v\rangle_{v\in V}\in NE(G)$ and ${\mathbf t}=\langle t_v\rangle_{v\in V}\in NE(G)$ such that ${\mathbf s}=_{{\cal B}(U),{\cal B}(W),B}{\mathbf t}$. We will prove that ${\mathbf s}=_C {\mathbf t}$. Indeed, consider strategy profile ${\mathbf e}=\langle e_v\rangle_{v\in V}$ such that 
$$
e_v=
\left\{
\begin{array}{ll}
s_v   & \mbox{if $v\in U$,}  \\
t_v  & \mbox{if $v\in W$}.  
\end{array}
\right.
$$ 
We will first prove that ${\mathbf e}\in NE(G)$. Assuming the opposite, let $v\in V$ be a player in the game $G$ that can increase his pay-off by changing strategy in profile ${\mathbf e}$. Without loss of generality, let $v\in U$. Then, ${\mathbf e}=_{Adj(v)\cup{\{v\}}}{\mathbf s}$. Thus, player $v$  can also increase his pay-off by changing strategy in profile ${\mathbf s}$, which is a contradiction with the choice of ${\mathbf s}\in NE(G)$.

Note that ${\mathbf e}=_{U,B}{\mathbf s}$ and ${\mathbf e}=_{W,B}{\mathbf t}$. Thus, ${\mathbf e}=_{A,B}{\mathbf s}$ and ${\mathbf e}=_{C}{\mathbf s}$. Hence, ${\mathbf e}=_{C}{\mathbf s}$ by the assumption $G\vDash A,B\rhd C$. Therefore, ${\mathbf s}=_{C}{\mathbf e} =_{C}{\mathbf t}$.
\end{proof}

\section{Completeness}

\begin{lemma}\label{border union}
${\cal B}(X\cup Y)\subseteq {\cal B}(X)\cup {\cal B}(Y)$.
\end{lemma}
\begin{proof}
Let $v\in {\cal B}(X\cup Y)$. Thus, $v\in X\cup Y$ and there is $w\notin X\cup Y$ such that $(v,w)\in E$. Without loss of generality, assume that $v\in X$. Hence, $v\in X$ and $w\notin X$. Therefore, $v\in {\cal B}(X)$. 
\end{proof}

\begin{lemma}\label{left mono}
$\vdash A\rhd C \rightarrow A,B \rhd C.$
\end{lemma}
\begin{proof}
Assume $A\rhd C$. By the Reflexivity axiom, $A,B\rhd A$. Thus, by the Transitivity axiom, $A,B\rhd C$. 
\end{proof}

\begin{lemma}\label{right mono}
$\vdash A\rhd B, C \rightarrow A \rhd B.$
\end{lemma}
\begin{proof}
Assume $A\rhd B, C$. By the Reflexivity axiom, $B,C\rhd B$. Thus, by the Transitivity axiom, $A\rhd B$. 
\end{proof}

\begin{theorem}\label{completeness}
For any graph $\Gamma=(V,E)$ and any formula $\phi\in \Phi(V)$, if $\nvdash_\Gamma \phi$, then there must exist a game $(V,\{S_v\}_{v\in V},\{u_v\}_{v\in V})$ over graph $\Gamma$ such that $G\nvDash \phi$.
\end{theorem}
\begin{proof}
Suppose that $\nvdash_\Gamma \phi$. Let $M$ be any maximal consistent subset of formulas in $\Phi(\Gamma)$ such that $\neg\phi\in M$.
\begin{definition}\label{A*}
For any set of vertices $A$, let $A^*$ be the set $\{v\in V\;|\; M \vdash A\rhd v\}$.
\end{definition}

\begin{theorem}\label{AsubA*}
$A\subseteq A^*$, for any $A\subseteq V$.
\end{theorem}
\begin{proof}
Let $a\in A$. By the Reflexivity axiom, $\vdash A\rhd a$. Hence, $a\in A^*$.
\end{proof}

\begin{lemma}\label{ArhdA*}
$M\vdash A\rhd A^*$, for any $A\subseteq V$.
\end{lemma}
\begin{proof}
Let $A^*=\{a_1,\dots,a_n\}$. By the definition of $A^*$,  $M\vdash A\rhd a_i$, for any $i\le n$. We will
prove, by induction on $k$, that $M\vdash (A\rhd a_1,\dots,a_k)$ for any $0\le k\le n$. 

\noindent {\em Base Case}: $M\vdash A\rhd \varnothing$ by the Reflexivity axiom.

\noindent {\em Induction Step}: Assume that $M\vdash (A\rhd a_1,\dots,a_k)$. By the Augmentation axiom, 
\begin{equation}\label{eq0}
M\vdash A, a_{k+1}\rhd a_1,\dots,a_k,a_{k+1}.
\end{equation}
Recall that $M\vdash A\rhd a_{k+1}$. Again by the Augmentation axiom, $M\vdash (A\rhd A, a_{k+1})$.
Hence, $M\vdash (A \rhd a_1,\dots,a_k,a_{k+1})$, by (\ref{eq0}) and the Transitivity axiom.
\end{proof}

For any set of vertices $A$, we will now define strategic game $$G_A=(V,\{S_v\}_{v\in V},\{u_v\}_{v\in V})$$ over graph $\Gamma$. For the purposes of this definition only, we assume that a direction is assigned to each edge of the graph $\Gamma$ in an arbitrary way.

Any player $v\in A^*$ may either choose strategy  $pass$ or opt to play ``pennies" with all of his adjacent players. In the latter case, he decides on either $heads$ or $tails$ for each adjacent player. The player cannot choose to pass with one player and to play pennies with others. Formally, if $v\in A^*$, then 
$$S_v=\{pass\}\cup \{f\;|\; f: Adj(v) \rightarrow \{heads, tails\}\}.$$ 
Similarly, any player $v\notin A^*$ may either choose  between strategies $0$ and $1$ or decide to play  pennies with all of his adjacent players. Thus, if $v\notin A^*$, then
$$S_v=\{0,1\}\cup \{f\;|\; f: Adj(v)\rightarrow \{heads, tails\}\}.$$
Furthermore, it will be assumed that any isolated (one that has no adjacent vertices) vertex of the graph  is prohibited from playing the pennies game. Thus such vertices either have  a single strategy $pass$ or a set of just two strategies: 0 and 1.

We define the pay-off function of any player $v$ as the sum of rewards in individual pennies mini-games on the edges adjacent to $v$ or a possible penalty imposed on $v$ for not playing the pennies.

\noindent{\bf Penalty.} If there are $u,w\in Adj^+(v)\setminus A^*$ such that $u$ plays strategy $0$ and $w$ plays strategy $1$, then a penalty in the amount of 1 is imposed on $v$ unless $v$ plays pennies.  

\begin{wrapfigure}{l}{0.45\textwidth}
\begin{center}
\vspace{0mm}
\scalebox{.5}{\includegraphics{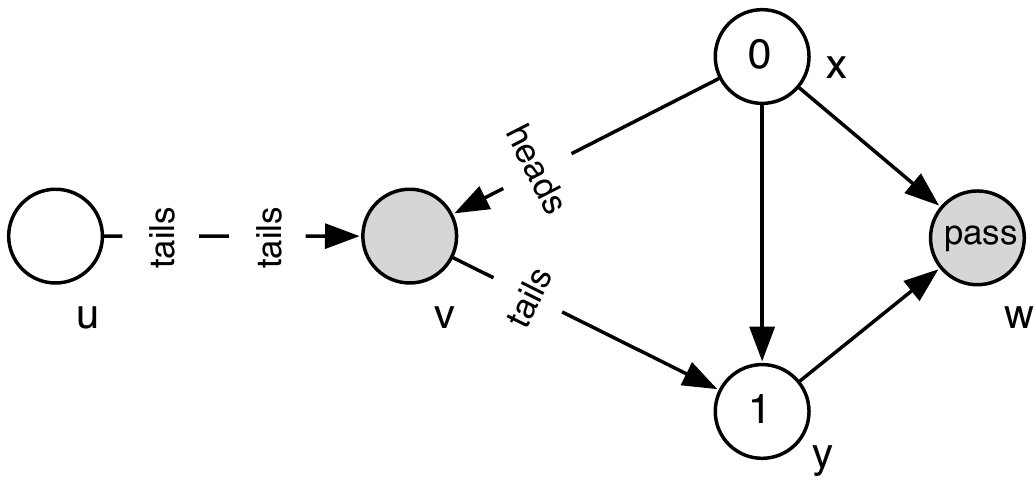}}
\vspace{0mm}
\footnotesize\caption{Strategy Profile}\label{payoff}
\vspace{-2mm}
\end{center}
\vspace{-1mm}
\end{wrapfigure}

For example, consider the strategy profile depicted in Figure~\ref{payoff}.  We will assume that elements of $A^*$ are the shaded vertices.  Vertices $w$, $x$, and $y$ are subject to the penalty because player $x$ plays 0 and player $y$ plays 1.  Players $u$ and $v$ have chosen to play pennies, and thus are not subject to any penalties.

\noindent{\bf Rewards.} In an individual mini-game along any edge of the graph $\Gamma$, rewards are only given if both players are playing pennies. The rewards are given according to the rules of a variation of the standard Matching Pennies game: the player who is, say, at the beginning of the directed edge is rewarded 1 for matching his opponent's strategy and the player at the opposite end of the edge is rewarded 1 for not matching his opponent's strategy.

For example, in the strategy profile depicted in Figure~\ref{payoff}, player $u$ gets the reward 1 for matching player $v$. Player $v$ is not rewarded. Players $x$, $y$, and $w$ also do not receive any rewards since they are not playing pennies. 

This concludes the definition of the game $G_A$.


\begin{lemma}\label{ne G_A}
In any Nash equilibrium of the game $G_A$, no player is playing the pennies game.
\end{lemma}
\begin{proof}
Assume that a vertex $v$ is playing pennies games in a strategy profile $\mathbf s$. Due to the definition of the set of strategies in the game $G_A$, vertex $v$ can not be an isolated vertex in the graph $\Gamma$. Let $u$ be any vertex adjacent to $v$. If player $u$ is not playing pennies in ${\mathbf s}$, then he can increase his pay-off by playing pennies. If player $u$ is playing pennies in ${\mathbf s}$, then either player $u$ or player $v$ would want to switch his strategy in the mini-game along edge $(u,v)$ since the two-player Matching Pennies game has no Nash equilibria. Therefore, strategy profile $\mathbf s$ is not a Nash equilibrium.
\end{proof}

\begin{lemma}\label{river}
$s_{w_1}= s_{w_2}$ for each $w_1,w_2\in Adj^+(v)\setminus A^*$, each $v\in V$, and each ${\mathbf s}=\langle s_w\rangle_{w\in V}\in NE(G_A)$.
\end{lemma}
\begin{proof}
By Lemma~\ref{ne G_A}, no player is playing pennies in profile $\mathbf s$. Suppose that $s_{w_1}\neq s_{w_2}$ for some $w_1,w_2\in Adj^+(v)\setminus A^*$. Thus, player $v$ is subject to penalty in the strategy profile $\mathbf s$. Then, he can increase his pay-off by starting to play pennies and avoiding the penalty. Therefore, ${\mathbf s}\notin NE(G_A)$.
\end{proof}



\begin{definition}\label{}
For any $u,v\in V$, let $u\sim v$ if there is a path from $u$ to $v$ in graph $\Gamma$ such that no two consecutive vertices of the path belong to the set $A^*$.
\end{definition}

\begin{lemma}\label{}
Relation $\sim$ is an equivalence relation on the set $V$. \qed
\end{lemma}
By $[v]$ we will denote the equivalence class of vertex $v$ with respect to this relation.

\begin{lemma}\label{sim}
If $u\sim v$, then $s_u=s_v$, for each $u,v\notin A^*$ and each  $\langle s_w\rangle_{w\in V}\in NE(G_A)$.
\end{lemma}
\begin{proof}
Let $\pi=(w_0,w_1,\dots,w_k)$ be a path connecting vertices $u$ and $v$ ($w_0=u$ and $w_k=v$) such that no two consecutive vertices in $\pi$ belong to $A^*$. We prove the statement by induction on $k$. If $k=0$ then $u=v$. Thus, $s_u=s_v$. Suppose now that $k>0$.

\noindent{Case I:} $w_1\notin A^*$. Thus, by Lemma~\ref{river}, $s_u=s_{w_1}$. By the Induction Hypothesis, $s_{w_1}=s_v$. Therefore, $s_u=s_v$.

\noindent{Case II:} $w_1\in A^*$. Thus, $w_1\neq v$ and, since no two consecutive vertices in $\pi$ belong to $A^*$, we have $w_2\notin A^*$. Hence, $s_u=s_{w_2}$ by Lemma~\ref{river} and $s_{w_2}=s_v$ by the Induction Hypothesis. Therefore, $s_u=s_v$.
\end{proof}

\begin{lemma}\label{border A*}
${\cal B}([v])\subseteq A^*$ for each $v\in V$.
\end{lemma}
\begin{proof}
Let $w\in {\cal B}([v])$, but $w\notin A^*$. By Definition~\ref{border}, there is $u\notin [v]$ such that $(w,u)\in E$. Consider the two-vertex path $\pi=(w,u)$. Since $w\notin A^*$, no two consecutive vertices in $\pi$ belong to $A^*$. 
Hence, $w\sim u$. Note that $v\sim w$ by Definition~\ref{border}. Thus, $v\sim u$, which is a contradiction. 
\end{proof}

\begin{lemma}\label{CD}
If $C\rhd D\in M$, then $G_A\vDash C\rhd D$.
\end{lemma}
\begin{proof}
Let ${\mathbf s}=\langle s_v\rangle_{v\in V},{\mathbf s'}=\langle s'_v\rangle_{v\in V}\in NE(G_A)$ such that ${\mathbf s}=_C{\mathbf s'}$. It will be sufficient to show that $s_d=s'_d$ for each $d\in D$. Consider any $d\in D$. If $d\in A^*$, then player $d$ has only two options in the game $G_A$: to play pennies or to decide to $pass$. By Lemma~\ref{ne G_A}, player $d$ chooses the strategy $pass$ under strategy profiles  ${\mathbf s}$ and ${\mathbf s'}$. Therefore,
$s_d =s'_d$. We will now assume that $d\notin A^*$. 


If $d\sim c_0$ for some $c_0\in C\setminus A^*$, then, by Lemma~\ref{sim}, $s_d=s_{c_0}=s'_{c_0}=s'_d.$
We will now assume that $d\nsim c$ for each $c\in C\setminus A^*$. Thus, $C\setminus A^*\subseteq \bigcup_{v\notin [d]}[v]$. Consider cut $(\bigcup_{v\notin [d]}[v], [d])$. By the Contiguity axiom,
$$
C\setminus A^*, A^*\rhd d \rightarrow {\cal B}(\bigcup_{v\notin [d]}[v]), {\cal B}([d]), A^*\rhd d 
$$
Due to the assumption $M\vdash C\rhd D$ and Lemma~\ref{left mono} and Lemma~\ref{right mono},
$$
M\vdash C\setminus A^*, A^*\rhd d
$$
Thus,
$$
M\vdash {\cal B}(\bigcup_{v\notin [d]}[v]), {\cal B}([d]), A^*\rhd d 
$$
By Lemma~\ref{border union}, ${\cal B}(\bigcup_{v\notin [d]}[v])\subseteq \bigcup_{v\notin [d]}{\cal B}([v])$
Hence, by Lemma~\ref{left mono},
$$
M\vdash \bigcup_{v\notin [d]}{\cal B}([v]), {\cal B}([d]), A^*\rhd d 
$$
Then, by Lemma~\ref{border A*}, $M\vdash  A^*\rhd d $. Thus, by Lemma~\ref{ArhdA*} and the Transitivity axiom, 
$M\vdash  A\rhd d$. Therefore, $d\in A^*$, which is a contradiction.
\end{proof}

\begin{definition}\label{sepsilon}
For any $A\subseteq V$ and any $k\in \{0,1\}$, let strategy profile ${\mathbf s}^{k,A}$ be defined as
$$
s^{k,A}_v=
\left\{
\begin{array}{ll}
pass   & \mbox{if $v\in A^*$,}  \\
k   & \mbox{otherwise}.  
\end{array}
\right.
$$ 
\end{definition}

\begin{lemma}\label{sepsilon is NE}
${\mathbf s}^{k,A}\in NE(G_A)$ for each $k\in \{0,1\}$.
\end{lemma}
\begin{proof}
By the definition of the game $G_A$, no player is paying a penalty in the strategy profile ${\mathbf s}^{k,A}$. At the same time, no player can get a reward by unilaterally switching to playing pennies.
\end{proof}

\begin{lemma}\label{b notin A*}
If $G_A\vDash A\rhd b$, then  $b\in A^*$.
\end{lemma}
\begin{proof}
Assume that $b\notin A^*$.
By Lemma~\ref{AsubA*}, ${\mathbf s}^{0,A}=_A{\mathbf s}^{1,A}$. At the same time, $s^{0,A}_b=0\neq 1=s^{1,A}_b$ since $b\notin A^*$. Therefore, $G_A\nvDash A\rhd b$.
\end{proof}

The product construction below defines a way to combine several games played over the same graph into a single game. The pay-off for a given player in the combined game is the sum of his pay-offs in the individual games.

\begin{definition}\label{}
Let $G^i=(V,\{S^i_v\}_{v\in V},\{u^i_v\}_{v\in V})$ for $i\in I$ be any family of games over the same graph $\Gamma=(V,E)$. By $\prod_{i\in I}G^i$ we mean game $(V,\{S_v\}_{v\in V},\{u_v\}_{v\in V})$ such that
\begin{enumerate}
\item $S_v$ is the Cartesian product  $\prod_{i\in I}S^i_v$,
\item $u_v=\sum_{i\in I}u^i_p$.
\end{enumerate}
\end{definition}

\begin{lemma}\label{big small}
If $\langle\langle s^i_v\rangle_{i\in I}\rangle_{v\in V} \in NE(\prod_{i\in I}G^i)$, then 
$\langle s^{i_0}_v\rangle_{v\in V} \in NE(G^{i_0})$ for each $i_0\in I$. \qed
\end{lemma}

\begin{lemma}\label{small big}
If $\langle s^{i}_v\rangle_{v\in V} \in NE(G^{i})$ for each $i\in I$,
then $$\langle\langle s^i_v\rangle_{i\in I}\rangle_{v\in V} \in NE(\prod_{i\in I}G^i).$$ \qed
\end{lemma}

\begin{lemma}\label{iff}
If $\{G^i\}_{i\in I}$ is a family of games over a graph $\Gamma$ such that each of these games has a nonempty set of Nash equilibria, then $\prod_{i\in I}G^i\vDash C\rhd D$ if and only if for each $i\in I$, $G^i\vDash C\rhd D$.
\end{lemma}
\begin{proof}
$(\Rightarrow):$ Assume that ${\mathbf s}^{i_0}=\langle s^{i_0}_v\rangle_{v\in V}\in NE(G^{i_0})$ and ${\mathbf t}^{i_0}=\langle t^{i_0}_v\rangle_{v\in V}\in NE(G^{i_0})$ are such that ${\mathbf s}^{i_0}=_C{\mathbf t}^{i_0}$. We will show that ${\mathbf s}^{i_0}=_D{\mathbf t}^{i_0}$.

By the assumption of the lemma, for each $i\in I$, the game $G^i$ has at least one Nash equilibria. We denote an arbitrary one of them by ${\mathbf e}^i=\langle e^i_v\rangle_{v\in V}$. Consider strategy profiles ${\mathbf S}=\langle \langle s^i_v\rangle_{i\in I}\rangle_{v\in V}$ and ${\mathbf T}=\langle \langle t^i_v\rangle_{i\in I}\rangle_{v\in V}$ for the game $\prod_{i\in I}G^i$ such that
$$
s^i_v=
\left\{
\begin{array}{ll}
s^{i_0}_v   & \mbox{if $i=i_0$,}  \\
e^i_v   & \mbox{otherwise}.  
\end{array}
\right.
$$ 
$$
t^i_v=
\left\{
\begin{array}{ll}
t^{i_0}_v   & \mbox{if $i=i_0$,}  \\
e^i_v   & \mbox{otherwise}.  
\end{array}
\right.
$$ 
By Lemma~\ref{small big}, we have ${\mathbf S},{\mathbf T}\in NE(\prod_{i\in I}G^i)$. Note that ${\mathbf S}=_C{\mathbf T}$ due to the assumption ${\mathbf s}^{i_0}=_C{\mathbf t}^{i_0}$. Hence, ${\mathbf S}=_D{\mathbf T}$ by the assumption of the lemma. Therefore, ${\mathbf s}^{i_0}=_D{\mathbf t}^{i_0}$.
 
$(\Leftarrow):$ Consider Nash equilibria ${\mathbf S}=\langle \langle s^i_v\rangle_{i\in I}\rangle_{v\in V}$ and ${\mathbf T}=\langle \langle t^i_v\rangle_{i\in I}\rangle_{v\in V}$ of the game $\prod_{i\in I}G^i$ such that ${\mathbf S}=_C{\mathbf T}$. We will show that ${\mathbf S}=_D{\mathbf T}$. It will be sufficient to show that $ \langle s^i_v\rangle_{v\in V}=_D \langle t^i_v\rangle_{v\in V}$ for each $i\in I$. Indeed, $ \langle s^i_v\rangle_{v\in V}=_C \langle t^i_v\rangle_{v\in V}$ due to the assumption ${\mathbf S}=_C{\mathbf T}$. By Lemma~\ref{big small}, $ \langle s^i_v\rangle_{v\in V},\langle t^i_v\rangle_{v\in V}\in NE(G^i)$.  Therefore, $ \langle s^i_v\rangle_{v\in V}=_D \langle t^i_v\rangle_{v\in V}$ by the assumption of the lemma.
\end{proof}
\begin{lemma}\label{iff2}
For any $\psi\in \Phi(\Gamma)$,
$\psi\in M$ if and only if $\prod_{A\subseteq V}G_A\vDash\psi$.
\end{lemma}
\begin{proof}
Induction on the structural complexity of formula $\psi$. The case $\psi\equiv\bot$ follows from the assumption of consistency of the set $M$ and Definition~\ref{true}. The case $\psi\equiv\psi_1\rightarrow\psi_2$ follows from  maximality and consistency of the set M and Definition~\ref{true} in the standard way.  Assume now that $\psi\equiv E\rhd F$.

If $E\rhd F\in M$, then, by Lemma~\ref{CD}, $G_A\vDash E\rhd F$ for each $A\subseteq V$. By Lemma~\ref{sepsilon is NE}, each of the games $G_A$ has at least one Nash equilibria: ${\mathbf s}^{0,A}$. Therefore, $\prod_{A\subseteq V}G_A\vDash E\rhd F$ by Lemma~\ref{iff}.

If $\prod_{A\subseteq V}G_A\vDash E\rhd F$, then by Lemma~\ref{iff}, $G_E\vDash E\rhd F$. Hence, by Definition~\ref{true},  $G_E\vDash E\rhd f$ for each $f\in F$. Hence, by Lemma~\ref{b notin A*}, $f\in E^*$ for each $f\in F$. Thus, $F\subseteq E^*$. Note that $M\vdash E\rhd E^*$ by Lemma~\ref{ArhdA*}. Hence, $M\vdash E\rhd F$ by Lemma~\ref{right mono}. Therefore, $E\rhd F\in M$ due to maximality of $M$.
\end{proof}

To finish the proof of the theorem, recall that $\neg\phi\in M$. Thus, $\phi\notin M$ due to consistency of $M$. Therefore, by Lemma~\ref{iff2}, $\prod_{A\subseteq V}G_A\nvDash\phi$.
\end{proof}

\section{Conclusion}

In this paper, we have described a sound and complete logical system for functional dependence in strategic games over a fixed dependency graph. The dependency graph puts restrictions on the type of pay-off functions that can be used in the game. If no such restrictions are imposed, then the logical system for functional dependence in strategic games is just the set of original Armstrong axioms. This statement follows from our results since the absence of restrictions corresponds to the case of a complete (in the graph theory sense) dependency graph. In the case of a complete graph, the Contiguity axiom follows from the Armstrong axioms because for any cut $(U,W)$, the set ${\cal B}(U)\cup{\cal B}(W)$ is the set of all vertices in the graph.

\bibliography{../sp}

\pagebreak

\appendix

\end{document}